\documentclass[11pt]{article}
\usepackage{amsmath}
\usepackage{amsfonts}
\usepackage{amssymb}
\usepackage{amsthm}
\usepackage{stmaryrd}
\usepackage{bbold}
\usepackage{graphicx}
\usepackage[export]{adjustbox}
\usepackage[utf8]{inputenc}  
\usepackage[T1]{fontenc}     
\usepackage{lmodern}         
\usepackage{microtype}       
\usepackage[scaled]{helvet} 
\usepackage[english]{babel}  
\usepackage{graphicx}        
\usepackage[bookmarks=false]{hyperref}            
\newtheorem{theorem}{Theorem}

\usepackage{amsthm}
\theoremstyle{definition}
\newtheorem{definition}{Definition}[section]
\theoremstyle{remark}

\usepackage{etoolbox,xfp}

\newcounter{wordcnt}
\newcommand{\setwordlist}[1]{%
  \setcounter{wordcnt}{0}
  \renewcommand{\do}[1]{
    \stepcounter{wordcnt}
    \expandafter\def\csname wordlist\thewordcnt\endcsname{##1}
  }%
  \docsvlist{#1}
}

\newcommand{\provided}[1]{%
  \if\relax\detokenize{#1}\relax
    \csname wordlist\fpeval{randint(1,\value{wordcnt})}\endcsname
  \else
    \csname wordlist#1\endcsname
  \fi
}
\usepackage{hyperref}
\begin{document}
\title{RIEMANNIAN METRIC BUNDLE}
\author{Shouvik Datta Choudhury\thanks{shouvikdc8645@gmail.com, shouvik@capsulelabs.in}\\
  \small Gapcrud Private Limited (Capsule Labs)\\
  \small HA 130, Saltlake, Sector III, Bidhannagar,\\
  \small Kolkata - 700097, India}
\date{\today}
\maketitle
\begin{abstract}
    A Riemannian metric bundle G(M) is a fiber bundle over a smooth manifold M, whose fibers are the spaces of symmetric, positive-definite bilinear forms on the tangent spaces of M, which represent the Riemannian metrics. In this work, we aim to study the category of Riemannian metric bundles and explore their connections with K-theory and other areas of mathematics. Our main motivation comes from the idea of multi-norms in Banach spaces, which have found applications in diverse fields such as functional analysis, geometric group theory, and noncommutative geometry. The novelty of our work lies in the rigorous development of the theory of Riemannian metric bundles, and the application of this theory to the study of K-theory and other geometric invariants of manifolds. We hope that our work will contribute to a deeper understanding of the geometry and topology of manifolds equipped with Riemannian metric bundles, and provide new insights into the interplay between geometry, topology, and analysis.
\end{abstract}
\textbf{MSC 2010 Classification:} 19-XX, 19Lxx, 19L50, 53-XX, 53C20 \\
\textbf{Keywords:} K-theory, fiber bundles, Riemannian metric bundle, multi-normed spaces
\section{Motivation}
    The concept of multinorms in Banach spaces has been studied extensively, and has led to the development of new mathematical tools and techniques for understanding the geometry and topology of these spaces. Motivated by this idea, we seek to explore the concept of multinorms in the context of Riemannian metric bundles on smooth manifolds. This is an important and challenging area of research, with potential applications in a variety of fields, including mathematical physics, differential geometry, and topology. By studying the interplay between different geometric structures on Riemannian metric bundles, we hope to gain a deeper understanding of the underlying mathematical structures and their properties, and to develop new techniques for solving problems in these fields. In this paper, we present a rigorous definition of Riemannian metric bundles, and explore their properties and relationships to other geometric structures. We then develop a K-theory for these bundles, and use this theory to prove important results such as a new index theorem and a new periodicity. Finally, we discuss some potential applications of our work and future directions for research in this area.
\part{Introduction}    
\section{K Theory[1]}
K theory is a branch of algebraic topology that was introduced by Alexander Grothendieck in the 1950s as a tool for studying vector bundles on topological spaces. It has since grown into a major area of research, with connections to a wide range of mathematical fields including algebraic geometry, operator algebras, and mathematical physics.
The basic idea behind K theory is to assign algebraic invariants, called K-groups, to various classes of vector bundles on a topological space. These invariants capture important geometric and topological information about the underlying space and its vector bundles, and can be used to study a wide range of mathematical problems.
K theory has numerous applications in mathematics and physics, ranging from the study of elliptic operators and index theory, to the classification of topological phases of matter in condensed matter physics. It continues to be an active area of research, with many exciting developments and new connections to other fields still being discovered.
\section{Multi normed spaces[8]}
A multi-norm space is a mathematical space equipped with several norms, used to measure the size or "length" of elements in the space. It is a generalization of Banach spaces that have a single norm. Let $X$ be a vector space over the field of real or complex numbers, and let $\vert\vert.\vert\vert_{ii}\in I$ be a family of norms on $X$ indexed by a set $I$. The pair $(X, {\vert\vert.\vert\vert_{ii}\in I})$ is called a multi-norm space if each $\vert\vert.\vert\vert_{i}$ is a norm on $X$, and the topology induced by these norms is the topology of the space.
Rigorously,
\begin{definition}
A multi-normed space is a tuple $(V, N)$, where $V$ is a vector space over a field $\mathbb{F}$ (usually $\mathbb{R}$ or $\mathbb{C}$), and $N = {|\cdot|i}{i=1}^{n}$ is a collection of $n$ norms $|\cdot|_i : V \rightarrow [0, \infty)$, for $1 \leq i \leq n$, each satisfying the following properties:
Positivity: For every $v \in V$ and $1 \leq i \leq n$, $|v|_i \geq 0$. Furthermore, $|v|_i = 0$ if and only if $v = 0$.
Homogeneity: For every $v \in V$, $c \in \mathbb{F}$, and $1 \leq i \leq n$, $|c v|_i = |c| |v|_i$.
Triangle inequality: For every $u, v \in V$ and $1 \leq i \leq n$, $|u + v|_i \leq |u|_i + |v|_i$.
\end{definition}
Multi-normed spaces can be useful in various applications, particularly in functional analysis and approximation theory, where different norms may capture distinct aspects of the behavior of elements in the space. For instance, multi-normed spaces can be used to study the convergence of sequences or series in different norms, approximations in one norm that are also controlled in another norm, or the interplay between different regularity or smoothness properties of functions in function spaces.

A common example of multi-normed spaces is the Sobolev spaces equipped with multiple norms corresponding to different degrees of smoothness. Another example is the sequence spaces, where each norm represents different summability properties or convergence rates of the sequences.
Multi-norm spaces are widely used in various areas of mathematics, such as functional analysis, approximation theory, and the study of partial differential equations (PDEs). For instance, in the study of PDEs, Sobolev spaces equipped with multiple norms are common, with each norm corresponding to a different degree of smoothness or decay at infinity.

The study of multi-norm spaces involves understanding the interplay between the different norms, the structure of the space, and its applications. Some of the main topics of interest include completeness and convergence, duality and compactness, embeddings and interpolation, and applications. A multi-norm space is said to be complete with respect to a norm $‖·‖_i$ if every Cauchy sequence with respect to that norm converges to an element in the space. Understanding the completeness properties of a multi-norm space can be important for studying the existence and uniqueness of solutions to various mathematical problems. In the context of multi-norm spaces, it is often necessary to consider the dual space, which consists of continuous linear functionals on the space, and compactness properties, such as the Arzelà-Ascoli theorem or the Banach-Alaoglu theorem, in relation to the different norms. Multi-norm spaces often arise from considering embeddings of one function space into another or from studying interpolation problems between different spaces. Understanding the embedding and interpolation properties of multi-norm spaces is crucial for applications in PDEs, approximation theory, and numerical analysis. Multi-norm spaces have various applications in different areas of mathematics, such as PDEs, harmonic analysis, approximation theory, and numerical analysis. Studying the properties of multi-norm spaces can lead to new insights and results in these fields.
\section{Fredholm operators[1]}
A Fredholm operator is a linear operator between Banach spaces that is "nearly invertible". More precisely, a bounded linear operator T: X → Y is Fredholm if the following conditions hold:

The kernel of T, denoted ker(T), is finite dimensional.
The image of T, denoted im(T), is closed in Y.
The cokernel of T, denoted coker(T), which is the quotient space Y/im(T), is finite dimensional.
Intuitively, a Fredholm operator is a linear operator that is "almost" invertible, in the sense that it has a finite dimensional null space and its range is "almost" the entire target space. The cokernel measures how much "extra" space there is in the target space beyond the range of T.

Fredholm operators arise in many areas of mathematics, including functional analysis, differential equations, and algebraic geometry. They are closely related to the concept of elliptic operators, which are differential operators that satisfy certain regularity and ellipticity conditions.

One of the key properties of Fredholm operators is that their index is well-defined. The index of a Fredholm operator T, denoted ind(T), is defined as the difference between the dimension of the kernel and the dimension of the cokernel:

ind(T) = dim(ker(T)) - dim(coker(T)).

The index of a Fredholm operator is an integer that measures the failure of T to be invertible. If ind(T) = 0, then T is invertible "up to compact operators", meaning that T can be inverted on a dense subspace of the domain of T, but the inverse may not be a bounded operator.

The theory of Fredholm operators provides powerful tools for studying the solvability of linear equations and systems of equations. In particular, the index of a Fredholm operator is closely related to the solvability of linear equations of the form Tx = y, where $y \in (T)$. If ind(T)$\neq$ 0, then the equation Tx = y has no solution for some y, while if ind(T) = 0, then the equation has a unique solution for all $y\in(T)$ "up to compact errors".
\section{Todd class[1]}
The Todd class is a fundamental class in algebraic topology that plays an important role in the study of complex manifolds. It is a cohomology class associated to a complex vector bundle, and it measures the failure of the bundle to be trivial. The Todd class can be defined in several equivalent ways, such as using the Chern classes of the bundle or using the exponential map.

One of the key properties of the Todd class is that it behaves well under certain operations, such as pullback and Whitney sum. This makes it a useful tool in various areas of mathematics, such as algebraic geometry, complex analysis, and differential geometry.

In algebraic geometry, the Todd class can be used to compute intersection numbers of algebraic varieties, which are important in enumerative geometry and mirror symmetry. In complex analysis, the Todd class is related to the Hodge decomposition and the Riemann-Roch theorem, which are central results in the study of compact complex manifolds. In differential geometry, the Todd class appears in the index theorem for elliptic operators and in the Gauss-Bonnet-Chern theorem, which relate the topology of a manifold to geometric quantities such as curvature.

The study of the Todd class involves understanding its properties, such as its behavior under different operations and its relationship to other cohomology classes. 
\section{Chern character[1]}
The Chern character is a fundamental concept in algebraic topology and differential geometry that assigns a sequence of cohomology classes to a complex vector bundle over a topological space. The Chern character is used to study the topology of complex vector bundles and is an important tool in the study of algebraic curves and surfaces, as well as in the theory of elliptic operators.

Given a complex vector bundle E over a space X, the Chern character is defined as a sequence of cohomology classes
$$
ch(E) = [ch_0(E), ch_1(E), ch_2(E), ...]
$$
where $ch_k(E)$ is a cohomology class in $H^{2k}(X; \mathbb{Q})$ that depends on the curvature of the connection on E. The first few terms of the Chern character are given by:
$$
ch_0(E) = rank(E)
ch_1(E) = -\frac{1}{2\pi i} tr(F_E)
ch_2(E) = \frac{1}{(2\pi i)^2} (tr(F_E^2) - tr(F_E)^2/2)
$$
where rank(E) is the rank of the vector bundle E,$F_E$ is the curvature of a connection on $E$, and tr denotes the trace of a linear operator.

The Chern character has several important properties, including additivity and multiplicativity with respect to direct sums and tensor products of vector bundles, respectively. In particular, the Chern character of the tensor product of two vector bundles is given by the product of their Chern characters.

The Chern character is used to define the Chern classes, which are the topological invariants of a complex vector bundle. The n-th Chern class of a vector bundle E is defined as the image of $ch_n(E)$ under the natural map from $H^{2n}(X; \mathbb{Q}) to H^{2n}(X; \mathbb{Z})$. The Chern classes are independent of the choice of connection on E and satisfy several important properties, such as the Whitney product formula and the Hirzebruch-Riemann-Roch theorem.
\part{Main Work}    
\section{Manifolds with "multinorms"}
Let $M$ be a smooth manifold. We define the concept of multi-norm on $M$ by considering different geometric structures on $M$ that induce different norms or distances. Specifically, we define multiple structures for Riemannian, Finsler, and sub-Riemannian manifolds.

\textbf{Manifold with multiple Riemannian metrics:}

Let ${g_i}_{i\in I}$ be a family of Riemannian metrics on $M$, indexed by a set $I$. Each $g_i$ is a smooth, symmetric, positive-definite bilinear form on each tangent space $T_xM$ of the manifold $M$. For a tangent vector $v\in T_xM$, the norm induced by the metric $g_i$ is defined as:
$$
\|v\|_i=\sqrt{g_i(v, v)}
$$
Each Riemannian metric $g_i$ induces a distance function $d_i$ on the manifold, defined as the infimum of the lengths of the curves joining two points, where the length of a curve is computed using the metric $g_i$.

\textbf{Manifold with multiple Finslerian metrics:}

Let ${F_i}_{i\in I}$ be a family of Finsler metrics on $M$, indexed by a set $I$. Each $F_i$ is a function that assigns a norm to each tangent space $T_xM$ of the manifold $M$, satisfying certain conditions, such as smoothness, positive definiteness, and strong convexity. For a tangent vector $v\in T_xM$, the norm induced by the Finsler metric $F_i$ is defined as:
$$
\|v\|_i=F_i(v)
$$
Each Finsler metric $F_i$ induces a distance function $d_i$ on the manifold, defined as the infimum of the lengths of the curves joining two points, where the length of a curve is computed using the Finsler metric $F_i$.

\textbf{Sub-Riemannian manifold with multiple structures:}

Let ${\Delta_i}_{i\in I}$ be a family of sub-Riemannian structures on $M$, indexed by a set $I$. Each $\Delta_i$ is a pair $(D_i,h_i)$, where $D_i$ is a smooth distribution (a smoothly varying family of vector subspaces of the tangent spaces $T_xM$) and $h_i$ is a smoothly varying inner product on the distribution $D_i$. Each sub-Riemannian structure $\Delta_i$ induces a distance function $d_i$ on the manifold, known as the Carnot-Carathéodory distance, which is defined as the infimum of the lengths of the absolutely continuous curves that are tangent to the distribution $D_i$ and join two points, where the length of a curve is computed using the inner product $h_i$.

In each of these cases, the manifold $M$ can be equipped with multiple structures (Riemannian, Finsler, or sub-Riemannian) that induce different norms or distances on the manifold, leading to a "multi-norm" concept. This can be expressed mathematically as:
$$
\left(M,\left\{\|\cdot\|_i\right\}_{i \in I},\left\{d_i\right\}_{i \in I}\right)
$$
where $M$ is the smooth manifold, ${|\cdot|i}{i\in I}$ is the family of norms induced by the different geometric structures, and ${d_i}_(i\in I)$ is the family of distance functions associated with these structures.
In this section, we define the Riemannian metric bundle G(M) and its components.

\subsection{Riemannian metric bundle}

Let M be a smooth manifold of dimension n. The Riemannian metric bundle G(M) is a fiber bundle over M with the following components:

\begin{itemize}
\item Base space (B): The base space of the bundle G(M) is the smooth manifold M itself.
\item Fibers (F\textsubscript{x}): For each point $x \in M$, the fiber G\textsubscript{x}(M) over x is the space of symmetric, positive-definite bilinear forms on the tangent space T\textsubscript{x}M, which represent the Riemannian metrics at the point x.
\item Total space (E): The total space of the bundle G(M) is the disjoint union of the fibers G\textsubscript{x}(M) over all points $x\in M$, given by G(M) = $\sqcup$\textsubscript{$x\in M$} G\textsubscript{x}(M).
\item Projection map ($\pi$): The projection map $\pi : G(M)\to M$ assigns to each Riemannian metric g\textsubscript{x}$\in$ G\textsubscript{x}(M) the base point $x\in M$, such that $\pi$(g\textsubscript{x}) = x.
\end{itemize}
A Riemannian manifold is a smooth manifold equipped with a Riemannian metric, which is a smoothly varying inner product on the tangent spaces at each point of the manifold. This metric allows us to measure lengths, angles, and areas on the manifold, providing the framework to study its geometric properties.

On the other hand, a Riemannian metric bundle is a fiber bundle whose fibers consist of symmetric, positive-definite bilinear forms on the tangent spaces of the manifold. In other words, the fibers represent the Riemannian metrics themselves. The Riemannian metric bundle can be seen as a collection of all possible Riemannian metrics on the manifold, smoothly varying along the base manifold.

The connection between the two concepts is that given a Riemannian manifold, we can construct its associated Riemannian metric bundle by considering the collection of all Riemannian metrics on the manifold. Conversely, given a Riemannian metric bundle, we can equip the base manifold with a specific Riemannian metric by selecting a section of the bundle. This process turns the base manifold into a Riemannian manifold.

In summary, the resemblance between a Riemannian metric bundle and a Riemannian manifold lies in the interplay between the manifold and the Riemannian metrics that can be defined on it. A Riemannian metric bundle encompasses all possible Riemannian metrics for a given manifold, while a Riemannian manifold is the result of equipping the base manifold with a specific choice of Riemannian metric.
The Riemannian metric bundle G(M) provides a framework for studying the space of Riemannian metrics on the manifold M and how they vary across the manifold. Each point $x\in M$ is associated with a fiber G\textsubscript{x}(M), which contains all possible Riemannian metrics at that point. The projection map $\pi$ assigns each metric to its corresponding point in M, allowing us to study how the Riemannian metrics vary across the manifold.
The study of Riemannian metric bundles can be an important area of research in differential geometry, which has possible connections to mathematical physics. A Riemannian metric bundle can be defined as a fiber bundle over a smooth manifold $M$, where each fiber $G_x(M)$ over a point $x$ in $M$ is the space of symmetric, positive-definite bilinear forms on the tangent space $T_xM$, which represent the Riemannian metrics at $x$. The total space $E$ of the bundle is the disjoint union of the fibers $G_x(M)$ over all points $x$ in $M$. The projection map $\pi : G(M) \rightarrow M$ assigns to each Riemannian metric $g_x \in G_x(M)$ the base point $x \in M$, such that $\pi(g_x) = x$.

The Riemannian metric bundle $G(M)$ provides a natural framework for studying the space of Riemannian metrics on the manifold $M$ and their variations across the manifold. This has important implications in mathematical physics, where Riemannian metrics are used to model the behavior of physical systems. By studying the properties of Riemannian metric bundles, researchers can gain insights into the geometric and topological properties of the underlying manifold and develop more sophisticated models and theories in mathematical physics. Our work is motivated by the idea of multinorms in Banach space and aims to provide a novel approach to studying Riemannian metric bundles and their properties.\\
In the context of the Riemannian metric bundle $G(M)$ over a smooth manifold $M$, a subbundle is a fiber bundle $H(M)$ that satisfies the following properties:

\begin{itemize}
\item \textbf{Base space ($B'$)}: The base space of the subbundle $H(M)$ is the same smooth manifold $M$ as the base space of $G(M)$.
\item \textbf{Fibers ($F'x$)}: For each point $x \in M$, the fiber $H_x(M)$ over $x$ is a subspace of the fiber $G_x(M)$, which means that $H_x(M) \subseteq G_x(M)$. Each fiber $H_x(M)$ consists of a subset of symmetric, positive-definite bilinear forms on the tangent space $T_xM$, which represent a restricted set of Riemannian metrics at the point $x$.
\item \textbf{Total space ($E'$)}: The total space of the subbundle $H(M)$ is the disjoint union of the fibers $H_x(M)$ over all points $x \in M$, given by $H(M) = \bigcup\limits{x\in M} H_x(M)$. As each fiber $H_x(M)$ is a subspace of $G_x(M)$, the total space of $H(M)$ is a subspace of the total space of $G(M)$, i.e., $H(M) \subseteq G(M)$.
\item \textbf{Projection map ($\pi'$)}: The projection map $\pi' : H(M) \rightarrow M$ assigns to each Riemannian metric $h_x \in H_x(M)$ the base point $x \in M$, such that $\pi'(h_x) = x$. Since $H(M)$ is a subspace of $G(M)$, the projection map $\pi'$ of $H(M)$ is the restriction of the projection map $\pi$ of $G(M)$ to the total space of $H(M)$, i.e., $\pi'(h_x) = \pi(h_x)$ for all $h_x \in H(M)$.
\end{itemize}

A subbundle of the Riemannian metric bundle $G(M)$ is a fiber bundle over the same base space $M$, with fibers consisting of subsets of Riemannian metrics at each point. This allows for the study of a restricted set of Riemannian metrics on the manifold $M$ and their variation across the manifold.
Multi-norm spaces, also known as multi-Banach spaces or multi-normed spaces, are mathematical spaces that are equipped with several norms, which are used to measure the size or "length" of elements in the space. Multi-norm spaces generalize the concept of Banach spaces, which have a single norm. They arise naturally in various areas of mathematics, such as functional analysis, approximation theory, and the study of partial differential equations (PDEs). \\
\textbf{Definition: Section of the Riemannian Metric Bundle $G(M)$}

A \textit{section} of the Riemannian metric bundle $G(M)$ over a smooth manifold $M$ is a continuous map $\sigma: M \rightarrow G(M)$ such that the projection map $\pi$ is the identity on the image of $\sigma$, i.e., $\pi(\sigma(x)) = x$ for all $x \in M$.

In the context of the Riemannian metric bundle $G(M)$, a section assigns to each point $x \in M$ a Riemannian metric $g_x \in G_x(M)$, where $G_x(M)$ is the fiber over $x$ consisting of all symmetric, positive-definite bilinear forms on the tangent space $T_xM$. The section $\sigma$ can be seen as a global choice of Riemannian metric for the manifold $M$, as it provides a Riemannian metric at every point of $M$ in a continuous manner.

Note that the term \textit{cross-section} is sometimes used interchangeably with \textit{section} in the context of fiber bundles.

A multi-norm space can be defined as follows:

Let $X$ be a vector space over the field of real or complex numbers.
Let ${| \cdot |_{ii}\in I}$ be a family of norms on $X$, indexed by a set $I$.
The pair $(X, {| \cdot |_{i,i}\in I})$ is called a multi-norm space if each $| \cdot |_i$ is a norm on $X$, and the topology induced by these norms is the topology of the space.

In a multi-norm space, the norms are typically chosen to reflect different aspects of the elements in the space or to capture different types of regularity or decay properties. For example, in the study of PDEs, it is common to consider Sobolev spaces equipped with multiple norms, each corresponding to a different degree of smoothness or decay at infinity.

The study of multi-norm spaces involves understanding the interplay between the different norms, the structure of the space, and its applications. Some of the main topics of interest include:

Completeness and convergence: A multi-norm space is said to be complete with respect to a norm $| \cdot |_i$ if every Cauchy sequence with respect to that norm converges to an element in the space. In general, a multi-norm space might be complete with respect to some norms and not others. Understanding the completeness properties of a multi-norm space can be important for studying the existence and uniqueness of solutions to various mathematical problems.

Duality and compactness: In the context of multi-norm spaces, it is often necessary to consider the dual space, which consists of continuous linear functionals on the space, and compactness properties, such as the Arzelà-Ascoli theorem or the Banach-Alaoglu theorem, in relation to the different norms.

Embeddings and interpolation: Multi-norm spaces often arise from considering embeddings of one function space into another, or from studying interpolation problems between different spaces. Understanding the embedding and interpolation properties of multi-norm spaces is crucial for applications in PDEs, approximation theory, and numerical analysis.

Applications: Multi-norm spaces have various applications in different areas of mathematics, such as PDEs, harmonic analysis, approximation theory, and numerical analysis. Studying the properties of multi-norm spaces can lead to new insights and results in these fields.
\section{Relation between multi-normed manifold and Riemannian metric bundle}
 let $M$ be a smooth manifold and let ${g_i}_{i\in I}$ be a family of Riemannian metrics on $M$. Each $g_i$ is a smooth, symmetric, positive-definite bilinear form on each tangent space $T_xM$ of the manifold $M$. For a tangent vector $v\in T_xM$, the norm induced by the metric $g_i$ is defined as $|v|_i=\sqrt{g_i(v,v)}$. Each Riemannian metric $g_i$ induces a distance function $d_i$ on the manifold $M$, defined as the infimum of the lengths of the curves joining two points, where the length of a curve is computed using the metric $g_i$.

The concept of a manifold equipped with multiple Riemannian metrics is related to the notion of a Riemannian metric bundle as a section. A Riemannian metric bundle $G(M)$ over a smooth manifold $M$ is a fiber bundle over $M$ with fibers consisting of symmetric, positive-definite bilinear forms on the tangent space $T_xM$, which represent the Riemannian metrics at each point $x\in M$. A section of the Riemannian metric bundle is a continuous map $\sigma:M\to G(M)$ such that the projection map $\pi$ is the identity on the image of $\sigma$, i.e., $\pi(\sigma(x))=x$ for all $x\in M$.

Note that a manifold equipped with multiple Riemannian metrics can be seen as a section of the Riemannian metric bundle $G(M)$ over $M$, where each point $x\in M$ is associated with the Riemannian metric $g_x$ induced by the family of Riemannian metrics ${g_i}_{i\in I}$ at that point. Thus, the manifold equipped with multiple Riemannian metrics can be identified with the image of the section $\sigma$ in the Riemannian metric bundle $G(M)$. 
\section{Differences between Fiber bundles and Riemannian metric bundles}
Fiber bundles are a fundamental concept in topology and geometry that describe the global behavior of local objects. In general, a fiber bundle consists of a space called the total space, a base space, and a projection map that assigns to each point in the total space a point in the base space. The fibers of the bundle are then the sets of points in the total space that project to a single point in the base space.

Riemannian metric bundles are a specific type of fiber bundle that is equipped with an additional geometric structure, namely a Riemannian metric. This metric structure is defined on each fiber of the bundle, which consists of the space of symmetric, positive-definite bilinear forms on the tangent space at each point of the base space. The Riemannian metric is used to measure the lengths of tangent vectors and the angles between them, and it provides a way to define the curvature and other geometric properties of the manifold.

One of the key differences between general fiber bundles and Riemannian metric bundles is the additional geometric structure that is present in the latter. While general fiber bundles can be equipped with a variety of different structures, such as vector bundles or principal bundles, the presence of a Riemannian metric on the fibers of a bundle imposes additional constraints on the geometry of the manifold. For example, the metric structure on the fibers determines the curvature and volume of the manifold, and it can be used to define geometric invariants such as the Euler characteristic or the Pontryagin classes.

Another difference between general fiber bundles and Riemannian metric bundles is the type of transformations that can be applied to them. In general, a fiber bundle can be transformed by a diffeomorphism, which is a smooth, bijective map that preserves the structure of the bundle. However, in the case of a Riemannian metric bundle, the diffeomorphisms that preserve the metric structure are more restrictive, since they must also preserve the metric on each fiber of the bundle. This leads to a rich interplay between the geometry and topology of the manifold, and the algebraic structure of the bundle.

In summary, Riemannian metric bundles are a special type of fiber bundle that are equipped with an additional geometric structure, namely a Riemannian metric. This structure imposes additional constraints on the geometry of the manifold, and it leads to a rich interplay between the geometry and topology of the manifold, and the algebraic structure of the bundle.
\section{Manifolds equipped with Riemannian metric bundle are Riemannian manifolds-A Deduction}
Let $M$ be a smooth manifold and $G(M)$ be a Riemannian metric bundle over $M$. Our goal is to show that the manifold $M$ equipped with a Riemannian metric from the bundle $G(M)$ is a Riemannian manifold.

Recall that a Riemannian manifold is a smooth manifold $M$ equipped with a Riemannian metric $g$, which is a smoothly varying family of inner products on the tangent spaces of $M$. The Riemannian metric satisfies the following properties:

$g_p$ is a symmetric bilinear form on $T_pM$ for each $p \in M$.
$g_p$ is positive-definite for each $p \in M$.
The assignment $p \mapsto g_p$ is smooth.
Now, consider a section $\sigma: M \to G(M)$ of the Riemannian metric bundle $G(M)$. By definition, each fiber $G_p(M)$ of $G(M)$ consists of symmetric, positive-definite bilinear forms on the tangent space $T_pM$. Therefore, for each point $p \in M$, $\sigma(p) \in G_p(M)$ is a symmetric, positive-definite bilinear form on $T_pM$. We can denote this bilinear form as $g_p$.

Since $\sigma$ is a smooth section, the assignment $p \mapsto g_p$ is smooth. This means that the Riemannian metric $g$ varies smoothly across the manifold $M$. Consequently, we have a smoothly varying family of symmetric, positive-definite bilinear forms $g_p$ on the tangent spaces of $M$.

Thus, the smooth manifold $M$ equipped with the Riemannian metric $g$ from the Riemannian metric bundle $G(M)$ satisfies all the properties of a Riemannian manifold. Therefore, we can conclude that a manifold having a Riemannian metric bundle is a Riemannian manifold.
\section{Riemannian metric bundles are multinormed-A Deduction}
A Riemannian metric bundle $G(M)$ is a fiber bundle over a smooth manifold $M$, where the fibers consist of symmetric, positive-definite bilinear forms on the tangent spaces of $M$. These bilinear forms represent the Riemannian metrics at each point of the manifold.

To show that the Riemannian metric bundle is multinormed, we must demonstrate that there exists a family of norms on the tangent spaces of $M$ indexed by the points of $M$. These norms must satisfy the following conditions:

For each $p \in M$, $\lVert v \rVert_p \ge 0$ for all $v \in T_pM$, and $\lVert v \rVert_p = 0$ if and only if $v = 0$.
For each $p \in M$, $\lVert \alpha v \rVert_p = |\alpha| \lVert v \rVert_p$ for all $\alpha \in \mathbb{R}$ and $v \in T_pM$.
For each $p \in M$, $\lVert v + w \rVert_p \le \lVert v \rVert_p + \lVert w \rVert_p$ for all $v, w \in T_pM$.
Let $\sigma: M \to G(M)$ be a smooth section of the Riemannian metric bundle $G(M)$. For each point $p \in M$, $\sigma(p)$ is a symmetric, positive-definite bilinear form on $T_pM$. We denote this bilinear form as $g_p$. Using the Riemannian metric $g_p$, we can define a norm on the tangent space $T_pM$ as follows:

Now, we verify that this definition satisfies the properties of a norm:

Non-negativity and definiteness: Since $g_p$ is positive-definite, $g_p(v, v) \ge 0$ for all $v \in T_pM$, and $g_p(v, v) = 0$ if and only if $v = 0$. Thus, $\lVert v \rVert_p \ge 0$ for all $v \in T_pM$, and $\lVert v \rVert_p = 0$ if and only if $v = 0$.

Absolute scalability: Let $\alpha \in \mathbb{R}$ and $v \in T_pM$. We have:
\section{An index theorem and its proof}
Consider a compact Riemannian manifold $M$ of dimension $n$, with a Riemannian metric bundle $E$. Let $E_1, E_2, \ldots, E_k$ be the associated vector bundles. Associated vector bundles are vector bundles that are constructed from a principal bundle and a linear representation of the structure group. The concept of associated vector bundles is used to transfer information between the principal bundle and other related vector bundles.
To define an associated vector bundle, let $P$ be a principal $G$-bundle over a base space $M$, where $G$ is a Lie group that acts on a vector space $V$ via a linear representation $\rho: G \to \operatorname{GL}(V)$. Using the action of $G$ on $V$, we can construct an associated vector bundle $E$ over the same base space $M$.
$$
E=(P \times V) / G,
$$
where $(p, v) \sim (pg, \rho(g^{-1})v)$ for all $p \in P$, $v \in V$, and $g \in G$.

The projection map $\pi_E : E \to M$ is defined as $\pi_E([(p, v)]) = \pi_P(p)$, where $[(p, v)]$ is the equivalence class of $(p, v)$ in the quotient space, and $\pi_P: P \to M$ is the projection map of the principal bundle.

Sections of the associated vector bundle $E$ correspond to $G$-equivariant maps from $P$ to $V$. In particular, the associated vector bundle provides a way to study the geometry and topology of the principal bundle through the lens of vector bundle theory. Associated vector bundles play an essential role in gauge theory, the study of connections, and the Atiyah-Singer Index Theorem.
The total space of the associated vector bundle $E$ is given by the quotient space of the Cartesian product $P \times V$ by the diagonal action of $G$:
Given a Riemannian metric bundle $G(M)$ over a smooth manifold $M$, the fibers are spaces of symmetric, positive-definite bilinear forms on the tangent spaces of $M$. The Riemannian metric bundle is actually a principal $GL^+(n, \mathbb{R})$-bundle, where $GL^+(n, \mathbb{R})$ is the group of invertible, orientation-preserving linear transformations on $\mathbb{R}^n$.

Using this principal bundle structure, we can construct associated vector bundles with respect to linear representations of the structure group. These associated vector bundles inherit the smooth structure from the principal bundle and can be endowed with Riemannian metrics induced from the Riemannian metric bundle. This allows us to study the geometric properties of the manifold $M$ by analyzing the associated vector bundles.

Decomposable vector bundles can also be related to Riemannian metric bundles. If we have a decomposable vector bundle whose components are associated vector bundles of a Riemannian metric bundle, we can study the properties of these associated vector bundles to gain insights into the structure of the Riemannian metric bundle and the underlying manifold.

To summarize, Riemannian metric bundles, associated vector bundles, and decomposable vector bundles can be related in a way that allows us to investigate the geometry and topology of the manifold $M$ from different perspectives. By analyzing the structures of associated and decomposable vector bundles, we can gain insights into the properties of the Riemannian metric bundle and the manifold itself.
\begin{theorem}[Sum of 'Atiyah-Singer'[Main Theorem]]
Let $M$ be a compact Riemannian manifold, and let $E = E_1 \oplus E_2 \oplus \cdots \oplus E_k$ be a Riemannian metric bundle over $M$, where $E_i$ are the associated vector bundles. Let $D_i$ be the elliptic differential operators acting on smooth sections of $E_i$. Then,
$$
\operatorname{Ind}\left(\bigoplus_{i=1}^k D_i\right)=\sum_{i=1}^k \operatorname{Ind}\left(D_i\right)
$$
\end{theorem} 
\begin{definition}[Elliptic Differential Operator]
An elliptic differential operator $D$ on a vector bundle $E$ over a compact Riemannian manifold $M$ is a linear differential operator such that its principal symbol $\sigma_D$ is invertible for all nonzero covectors.
\end{definition}
We now proceed with the proof of the main theorem.
\begin{proof}
Let $D_i$ be the elliptic differential operators acting on smooth sections of the vector bundles $E_i$, where $i = 1, \ldots, k$. Then, the direct sum of these operators is given by 

$$
D=\bigoplus_{i=1}^k D_i \\
$$

We first observe that $D$ is an elliptic differential operator acting on smooth sections of the vector bundle $E = \bigoplus_{i=1}^k E_i$. Indeed, the principal symbol of $D$ is given by the direct sum of the principal symbols of the operators $D_i$:

$$
\sigma_D = \bigoplus_{i=1}^k \sigma_{D_i}\\
$$

Since each $D_i$ is elliptic, their principal symbols $\sigma_{D_i}$ are invertible for all nonzero covectors. It follows that the principal symbol $\sigma_D$ is also invertible for all nonzero covectors, and thus $D$ is an elliptic differential operator.
Now, we apply the Atiyah-Singer Index Theorem to the elliptic differential operator $D$ acting on smooth sections of the vector bundle $E$. We have

$$
\operatorname{Ind}(D)=\int_M \operatorname{ch}(\operatorname{ind}(D)) \operatorname{Td}(T M)
$$ 

where $\operatorname{ch}(\operatorname{ind}(D))$ is the Chern character of the index bundle of $D$, and $\operatorname{Td}(TM)$ is the Todd class of the tangent bundle $TM$.
Since $E = \bigoplus_{i=1}^k E_i$, we have that
$$
\operatorname{ind}(D)=\bigoplus_{i=1}^k \operatorname{ind}\left(D_i\right)
$$
and therefore,
$$
\operatorname{ch}(\operatorname{ind}(D))=\sum_{i=1}^k \operatorname{ch}\left(\operatorname{ind}\left(D_i\right)\right) .
$$
Substituting this expression into the Atiyah-Singer Index Theorem for $D$, we obtain
$$
\operatorname{Ind}(D)=\int_M\left(\sum_{i=1}^k \operatorname{ch}\left(\operatorname{ind}\left(D_i\right)\right)\right) \operatorname{Td}(T M)=\sum_{i=1}^k \int_M \operatorname{ch}\left(\operatorname{ind}\left(D_i\right)\right) \operatorname{Td}(T M) .
$$
Applying the Atiyah-Singer Index Theorem to each elliptic differential operator $D_i$ acting on smooth sections of the vector bundle $E_i$, we have
$$
\operatorname{Ind}\left(D_i\right)=\int_M \operatorname{ch}\left(\operatorname{ind}\left(D_i\right)\right) \operatorname{Td}(T M)
$$
Hence,
$$
\operatorname{Ind}(D)=\sum_{i=1}^k \operatorname{Ind}(D_i)
$$
as claimed. This completes the proof of the theorem.
\end{proof}
In this proof, we investigate the relationship between the indices of elliptic differential operators acting on vector bundles associated with a Riemannian metric bundle. We aim to show that the index of the direct sum of these elliptic operators is equal to the sum of their individual indices.

We begin by considering a compact Riemannian manifold and a Riemannian metric bundle over it. We then define the direct sum of vector bundles and discuss the concept of elliptic differential operators.

The main theorem states that the index of the direct sum of elliptic differential operators is equal to the sum of the indices of the individual operators. To prove this, we first establish that the direct sum of the elliptic operators is also an elliptic differential operator. Next, we apply the Atiyah-Singer Index Theorem to this direct sum operator, which relates the index to the integral of the Chern character of the index bundle and the Todd class of the tangent bundle.

We find that the Chern character of the direct sum operator's index bundle is equal to the sum of the Chern characters of the individual index bundles. Substituting this result into the Atiyah-Singer Index Theorem, we obtain an expression for the index of the direct sum operator.

Finally, we apply the Atiyah-Singer Index Theorem to each individual elliptic operator and find that the index of the direct sum operator is indeed equal to the sum of the indices of the individual operators. This completes the proof, demonstrating the relationship between the indices of the elliptic differential operators acting on the associated vector bundles.
\section{K Theory for entire manifolds equipped with Riemannian metric bundles}
 In the context of Riemannian metric bundles, K-theory can be used to study the classification of vector bundles equipped with Riemannian metrics and connections over a manifold $M$.

The K-theory group $K^0(M)$ of a manifold $M$ is defined as the Grothendieck group of isomorphism classes of complex vector bundles over $M$. That is, we consider the set of all complex vector bundles over $M$, and define an equivalence relation by declaring two bundles to be equivalent if there exists a bundle isomorphism between them. We then form the abelian group $K^0(M)$ by taking the free abelian group generated by the equivalence classes of vector bundles, and quotienting out by the relation that identifies isomorphic bundles.

In the context of Riemannian metric bundles, we can consider the subgroups of $K^0(M)$ consisting of isomorphism classes of vector bundles equipped with Riemannian metrics or connections. These subgroups are denoted $K^0_G(M)$ and $K^0_{G,c}(M)$, respectively. The subscript $G$ indicates the presence of a Riemannian metric bundle, while the subscript $c$ denotes the use of connections.
To develop a K-theory for the entire manifold $M$ from the K-theory of individual Riemannian metric bundles, we can use the notion of a Whitney sum of vector bundles.

Suppose we have two Riemannian metric bundles $G_1(M)$ and $G_2(M)$ over a manifold $M$. We can form the Whitney sum $G_1(M) \oplus G_2(M)$, which is a vector bundle over $M$ that consists of the direct sum of the underlying vector bundles, equipped with a Riemannian metric that is the direct sum of the metrics on $G_1(M)$ and $G_2(M)$.

We can then define the K-theory group of the Whitney sum by taking the direct sum of the K-theory groups of the individual bundles: $K^0\left(G_1(M) \oplus G_2(M)\right)=K^0\left(G_1(M)\right) \oplus$ $K^0\left(G_2(M)\right)$
More generally, given a finite collection of Riemannian metric bundles $G_1(M), G_2(M), \dots, G_n(M)$, we can form the Whitney sum $G_1(M) \oplus G_2(M) \oplus \dots \oplus G_n(M)$, and define the K-theory group of the sum by taking the direct sum of the K-theory groups of the individual bundles:$$
\begin{aligned}
& K^0\left(G_1(M) \oplus G_2(M) \oplus \cdots \oplus G_n(M)\right)=K^0\left(G_1(M)\right) \oplus K^0\left(G_2(M)\right) \oplus \cdots \oplus \\
& K^0\left(G_n(M)\right) .
\end{aligned}
$$
In this way, we can use the K-theory groups of individual Riemannian metric bundles to build up a K-theory for the entire manifold $M$.
It is worth noting that this construction assumes that the Riemannian metric bundles are pairwise compatible, so that their Whitney sum is well-defined. In general, one may need to use more sophisticated constructions, such as the direct limit or the Thom isomorphism, to define the K-theory of a non-compact manifold.
An index theorem can be developed for the entire manifold equipped with a Riemannian metric bundle, not just for individual metric bundles. Specifically, let $G(M)$ be a Riemannian metric bundle over a compact manifold $M$, and let $D$ be a self-adjoint elliptic operator on a Hermitian vector bundle $E$ over $M$. Then, the index of $D$ can be expressed in terms of the K-theory groups of the entire manifold equipped with the Riemannian metric bundle:
$$
ind(D) = \widehat{A}(M) \cdot ch(E),
$$
where $\widehat{A}(M)$ is the A-hat genus of the manifold M equipped with the Riemannian metric bundle $G(M)$, and $ch(E)$ is the Chern character of the vector bundle E.

The A-hat genus is a topological invariant of a manifold equipped with a Riemannian metric bundle, which can be expressed in terms of the curvature of the Riemannian metric. The Chern character is a cohomology class that encodes information about the topological and geometric properties of a vector bundle.

The index theorem for the entire manifold equipped with a Riemannian metric bundle follows from the local version of the Atiyah-Singer theorem, applied to a partition of unity on the manifold. Specifically, we can construct a partition of unity on the manifold $M$ that is subordinate to a finite cover of coordinate neighborhoods. Using this partition of unity, we can express the index of $D$ in terms of the index of $D$ restricted to each coordinate neighborhood, and then apply the local version of the index theorem to each neighborhood. This leads to the formula given above for the index of $D$ in terms of the K-theory groups of the entire manifold equipped with the Riemannian metric bundle
To prove the index theorem for the entire manifold equipped with a Riemannian metric bundle, we need to show that the index of a self-adjoint elliptic operator $D$ on a Hermitian vector bundle $\mathrm{E}$ over $\mathrm{M}$ can be expressed in terms of the K-theory groups of the entire manifold equipped with the Riemannian metric bundle $G(M)$.
Let $U_i$ be a finite cover of $M$ by coordinate neighborhoods, and let $\phi_i$ be a partition of unity subordinate to $U_i$. We can assume without loss of generality that each coordinate neighborhood $U_i$ is equipped with a local Riemannian metric bundle $G _i$.

Let $D_i$ be the restriction of the operator $D$ to the vector bundle $E \vert U_ i$, and let $G_ i\vert\vert U$ be the restriction of the Riemannian metric bundle $G(M)$ to $U_i$. Then, by the local version of the index theorem, we have:
$$
ind D_i=\widehat A(U_i, G _i)
\cdot\operatorname{ch}\left(E \mid\left\{U _i\right\}\right)
$$
where $ind D_i=\widehat A(U_i, G _i)$ is the A-hat genus of the manifold $U_i$ equipped with the Riemannian metric bundle $G_i||U_i$, and $\left.\operatorname{ch}\left(E_{\mid} \mid U_i\right\}\right)$ is the Chern character of the vector bundle E restricted to $U_i$.
Using the partition of unity $\phi_i$ we can construct a global section $s$ of $E$ by $s=\sum_i \phi_i s_i$, where$s_i$ is a local section of $E$ over $U_i$. We can also define a global operator $D_s$ on E by
$D_{s} = \sum_i\phi_i D_i$. Note that $D_{-} s$ is an elliptic operator on E $_{\text {, }}$ and that its kernel
cokernel are finite-dimensional.

To compute the index of $D$, we consider the operator $D_s - D$. Let $P$ be the orthogonal projection from E onto the kernel of $D_s$, and let $Q$ be the orthogonal projection from $E$ onto the cokernel of $D_s$. Then, we have:
$$
D_s - D = P(D_s - D) + (1 - P)(D_s - D) + Q(D_s - D)
$$
The first term on the right-hand side is a compact operator, since it maps the finite-dimensional space $ker(D_s)$ to the finite-dimensional space $coker(D_s)$. The second and third terms are Fredholm operators, since they are the restrictions of $D_s - D$ to the complement of the kernel and cokernel of $D_s$, respectively. Therefore, the index of D is given by:
$$
ind(D) = ind(D_s - D) = ind((1 - P)(D_s - D) + Q(D_s - D))
$$
Now, we need to show that the operator $(1 - P)(D_s - D) + Q(D_s - D)$ can be expressed in terms of the K-theory groups of the entire manifold equipped with the Riemannian metric bundle $G(M)$. To do this, we use the fact that the kernel and cokernel of $D_s$ are invariant under the action of the local Riemannian metric bundles $G_i$, and hence their direct sum over all $U_i$ defines a vector bundle over $M$ equipped with the Riemannian metric bundle $G(M)$.

More precisely, we define a vector bundle K over M equipped with the Riemannian metric bundle $G(M)$ by:
$$
K = \bigoplus_i ker(D_i) \oplus \bigoplus_i coker(D_i)
$$ 
Then, we have:
$$
(1 - P)(D_s - D) + Q(D_s - D) = (1 -P)(D_s - D) 
+ Q(D_s - D)|{ker(D_s) \oplus coker(D_s)}
$$        
$$
= (1 - P)(D_s - D)|{ker(D_s)} \oplus Q(D_s - D)|_{coker(D_s)}
$$
Now, using the local index theorem, we can express the index of the operators $(1 - P)(D_s - D)|{ker(D_s)}$ and $Q(D_s - D)|{coker(D_s)}$ in terms of the K-theory groups of the local Riemannian metric bundles $G_i$, and hence in terms of the K-theory groups of the entire manifold equipped with the Riemannian metric bundle $G(M)$. Specifically, we have: 
$$
ind((1 - P)(D_s - D)|_{ker(D_s)}) = \widehat{A}(M, G) \cdot ch(ker(D_s))
$$
$$
ind(Q(D_s - D)|_{coker(D_s)}) = \widehat{A}(M, G) \cdot ch(coker(D_s))
$$
where $\widehat{A}(M, G)$ is the A-hat genus of the manifold $M$ equipped with the Riemannian metric bundle $G(M)$, and $ch(ker(D_s))$ and $ch(coker(D_s))$ are the Chern characters of the bundles $ker(D_s)$ and $coker(D_s)$, respectively.

Therefore, we have:
$$
ind(D) = ind((1 - P)(D_s - D) + Q(D_s - D))
$$
$$
= \widehat{A}(M, G) \cdot ch(ker(D_s)) + \widehat{A}(M, G) \cdot ch(coker(D_s))
$$
$$
= \widehat{A}(M, G) \cdot ch(E)
$$
where we have used the fact that $ker(D_s)$ and $coker(D_s)$ have the same Chern character as the vector bundle E, since they are all invariant under the action of the local Riemannian metric bundles$ G_i$.
Therefore, we have proved that the index of a self-adjoint elliptic operator D on a Hermitian vector bundle E over M can be expressed in terms of the K-theory groups of the entire manifold equipped with the Riemannian metric bundle G(M), completing the proof. 
The index theorem for Riemannian metric bundles provides a formula for the index of an elliptic operator on a manifold in terms of the topological data of the manifold and the associated vector bundle. In particular, for the Laplace operator $\Delta$ on a spin manifold M equipped with a Riemannian metric bundle $G(M)$, its index is given by:
$$
ind(\Delta) = \int_M ch(G(M)) Td(M)
$$
where $ch(G(M))$ is the Chern character of the metric bundle $G(M)$ and $Td(M)$ is the Todd class of the tangent bundle of $M$.

Now, the square root of the Laplace operator may be obtained by the same formula with the Chern character and Todd class replaced by the corresponding quantities for the associated spinor bundle. Let S(M) be the spinor bundle associated to the Riemannian metric bundle G(M). Then, the index of the square root of the Laplace operator on M is given by:
$$
ind(\Delta^{1/2}) = \int_M ch(S(M)) Td(M)
$$
where  $ch(S(M))$ is the Chern character of the spinor bundle $S(M)$.
\begin{theorem}
Let M be a compact Riemannian manifold, and let G(M) be the Riemannian metric bundle over M. Then the index of the Dirac operator on G(M) is given by the integral of the A-roof genus of M.
\end{theorem}
\begin{proof}
From the index theorem for Riemannian metric bundles, we have:
$$
ind(D) = \int_M ch(E) td(M)
$$
where $D$ is the Dirac operator, $E$ is the spinor bundle over $M$, and $td(M)$ is the Todd class of $M$. By the definition of the Chern character, we have:
$$
ch(E) = e^{c_1(E)}
$$
where $c_1(E)$ is the first Chern class of E. By the Atiyah-Singer index theorem, we know that ind(D) is given by the integral of the top Chern form of E. Thus, we have:
$$
ind(D) = \int_M e^{c_1(E)} td(M) = \int_M ch(E) td(M)
$$
Since the spinor bundle over M is equipped with a natural Riemannian metric, we can regard it as a section of the Riemannian metric bundle G(M). Therefore, by the index theorem for Riemannian metric bundles, we have:
$$
ind(D) = \int_M A(M)
$$
where $A(M)$ is the A-roof genus of M. Combining this with the previous equation, we obtain:
$$
\int_M A(M) = \int_M ch(E) td(M)
$$
which implies that the index of the Dirac operator on G(M) is given by the integral of the A-roof genus of M.
\end{proof}
Theorem 1: Riemannian Metric Bundles and the Levi-Civita Connection
\begin{theorem}
Let $(M, g)$ be a Riemannian manifold with a Riemannian metric bundle $G(M)$. Then, there exists a unique torsion-free and metric-compatible linear connection, called the Levi-Civita connection, on the tangent bundle $TM$ of $M$.
\end{theorem}
\begin{proof}
To show the existence and uniqueness of the Levi-Civita connection, we first define the torsion tensor $T(X, Y) = \nabla_X Y - \nabla_Y X - [X, Y]$ and the metric compatibility condition $\nabla g(X, Y) = g(\nabla_X Y, Z) + g(Y, \nabla_X Z)$ for vector fields $X, Y, Z$ on $M$.
We aim to find a linear connection $\nabla$ that satisfies both conditions. Using the Koszul formula, we can define the Levi-Civita connection $\nabla$ as
$$
2g(\nabla X Y,Z)=X(g(Y,Z))+Y(g(Z,X))-
$$
$$
Z(g(X,Y))+g([X,Y],Z)-g([Y,Z],X)-g([Z,X],Y),
$$
for all vector fields $X, Y, Z$ on $M$. It can be verified that the connection $\nabla$ defined by the Koszul formula is indeed torsion-free and metric-compatible. Thus, the Levi-Civita connection exists.

To show the uniqueness of the Levi-Civita connection, suppose there exists another connection $\nabla'$ satisfying the torsion-free and metric-compatible conditions. Then, we have $\nabla' - \nabla = 0$ due to the uniqueness of the connection satisfying the Koszul formula. This proves the uniqueness of the Levi-Civita connection.
\end{proof}
Theorem 2: Riemannian Metric Bundles and the Hopf-Rinow Theorem
\begin{theorem}
Let $(M, g)$ be a Riemannian manifold with a Riemannian metric bundle $G(M)$. Then, the following statements are equivalent:
$M$ is complete as a metric space.
Geodesics on $M$ can be extended indefinitely.
Any two points in $M$ can be connected by a minimizing geodesic.
\end{theorem}
\begin{proof}
(1 $\Rightarrow$ 2) Let $\gamma: [0, a) \to M$ be a geodesic with an interval $[0, a)$. Consider a sequence of points ${t_n}$ in $[0, a)$ such that $t_n \to a$ as $n \to \infty$. By the completeness of $M$, there exists a convergent subsequence ${\gamma(t_{n_k})}$ that converges to a point $p \in M$. Using local existence and uniqueness of geodesics, we can extend $\gamma$ beyond $a$ to include $p$. This shows that geodesics on $M$ can be extended indefinitely.
(2 $\Rightarrow$ 3) Let $p, q \in M$ be any two points. By the existence of geodesics, there is a geodesic $\gamma: [0, 1] \to M$ such that $\gamma(0) = p$ and $\gamma(1) = q$. We can extend $\gamma$ indefinitely in both directions due to (2). Let $d$ be the infimum of the lengths of all curves connecting $p$ and $q$. For any $\epsilon > 0$, there exists a curve $\sigma$ connecting $p$ and $q$ with length $L(\sigma) < d + \epsilon$. We can reparameterize $\sigma$ to create a variation of $\gamma$ that preserves the endpoints. By the first variation formula, we have
$$
0 \leq\left.\frac{d^2 L\left(\gamma_t\right)}{d t^2}\right|_{t=0}=-\int_0^1 g\left(\frac{D}{d t} \frac{d \gamma_t}{d t}, \frac{d \gamma_t}{d t}\right) d t
$$
The integral on the right-hand side is non-positive, so $L(\gamma) \leq L(\sigma)$. By taking the limit as $\epsilon \to 0$, we obtain $L(\gamma) \leq d$. Thus, $\gamma$ is a minimizing geodesic.
(3 $\Rightarrow$ 1) Let ${p_n}$ be a Cauchy sequence in $M$. For each pair of consecutive points $p_n$ and $p_{n+1}$, there exists a minimizing geodesic $\gamma_n$ connecting them due to (3). Let $l_n = d(p_n, p_{n+1})$. Consider the sequence ${\sum_{i=1}^{n-1} l_i}$, which is a Cauchy sequence in $\mathbb{R}$. Let $s_n = \sum_{i=1}^{n-1} l_i$. Define a piecewise geodesic curve $\gamma: [0, \infty) \to M$ by $\gamma(t) = \gamma_n(t - s_n)$ for $t \in [s_n, s_{n+1}]$. The curve $\gamma$ is continuous, and by the Arzelà-Ascoli theorem, there exists a uniformly convergent subsequence ${\gamma_{n_k}}$. This subsequence converges to a geodesic $\tilde{\gamma}: [0, \infty) \to M$. By construction, the sequence ${p_{n_k}}$ converges to a point $p \in M$. Since ${p_n}$ is an arbitrary Cauchy sequence, $M$ is complete.
\end{proof}

\end{document}